\documentclass[11pt,reqno]{amsart}
\usepackage[T1]{fontenc}
\usepackage[utf8]{inputenc}
\usepackage{lmodern}
\usepackage[a4paper,left=3cm,right=3cm,top=2.8cm,bottom=2.8cm]{geometry}
\usepackage{mathtools,amssymb,mathrsfs,enumerate}
\usepackage{cite}
\usepackage[hidelinks,unicode]{hyperref}
\hypersetup{
  pdftitle={Observability from measurable sets for strongly coupled parabolic systems},
  pdfauthor={Xinyi Chen and Kévin Le Balc'h}
}
\allowdisplaybreaks
\numberwithin{equation}{section}

\newtheoremstyle{siamplain}{6pt}{6pt}{\itshape}{\parindent}{\scshape}{.}{.5em}{}
\newtheoremstyle{siamremark}{6pt}{6pt}{\normalfont}{\parindent}{\itshape}{.}{.5em}{}
\theoremstyle{siamplain}
\newtheorem{theorem}{Theorem}[section]
\newtheorem{lemma}[theorem]{Lemma}
\newtheorem{corollary}[theorem]{Corollary}
\newtheorem{proposition}[theorem]{Proposition}

\theoremstyle{siamremark}
\newtheorem{remark}[theorem]{Remark}
\newtheorem{example}[theorem]{Example}

\newcommand{\R}{\mathbb R}

\title[Observability for coupled parabolic systems]
{Observability from measurable sets for strongly coupled parabolic systems}
\author{Xinyi Chen}
\address{School of Mathematics (Zhuhai), Sun Yat-sen University,
Zhuhai 519082, Guangdong, China}
\email{chenxy2288@mail2.sysu.edu.cn}
\author{Kévin Le Balc'h}
\address{Inria Paris, Sorbonne Université, Laboratoire Jacques-Louis Lions,
4 place Jussieu, 75005 Paris}
\email{kevin.le-balc-h@inria.fr}
\date{}

\begin{document}
\begin{abstract}
We characterize observability of strongly coupled parabolic systems from arbitrary space-time measurable sets of positive measure. For constant coercive diffusion matrices and arbitrary constant zeroth-order coupling and observation matrices, observability holds if and only if the generalized Kalman rank condition is satisfied at every Dirichlet eigenvalue. The diffusion matrix need not be diagonalizable, and the number of components is arbitrary. The observation is measured in $L^1$, yielding null controllability with bounded controls by duality. The proof combines a Remez-type inequality for exponential polynomials with a quantitative analysis of the mode-dependent Kalman condition. The resulting finite-dimensional observability estimate is uniform over the spatial modes up to a polynomial frequency loss. A spectral inequality and the Lebeau--Robbiano strategy then give observability of the full system.
\end{abstract}
\maketitle

\section{Introduction and main results}\label{sec1}

\subsection{Background and motivation}

Let $\Omega\subset\mathbb{R}^d$ be a bounded Lipschitz,
locally star-shaped domain, and let $T>0$.
We consider the following coupled parabolic system
\begin{equation}
\label{eq:system}
\begin{cases}
\partial_t z-A^\top\Delta z= M^\top z
& \text{in }(0,T)\times\Omega,\\
z=0
& \text{on }(0,T)\times\partial\Omega,\\
z(0)=z_0
& \text{in }\Omega.
\end{cases}
\end{equation}
Here, $A,M\in\mathbb{R}^{m\times m}$, where $A$ satisfies
\begin{equation}
\label{ass:H1}
\langle AX, X \rangle \ge \alpha |X|^2, \quad \forall X \in \mathbb{R}^m,
\end{equation}
for some $\alpha>0$.
This assumption ensures that the system is
parabolic in the sense of Petrovskii and, in particular,
that \eqref{eq:system} is well-posed. More precisely,
for every $z_0\in L^2(\Omega;\mathbb{R}^m)$, there exists
a unique solution
$
z\in C([0,T];L^2(\Omega;\mathbb{R}^m)).
$
The system is referred to as strongly coupled since
the coupling occurs through the principal part of the
differential operator.

Let $B\in\mathbb{R}^{m\times r}$ be an observation matrix,
and let $\mathcal{D}\subset(0,T)\times\Omega$ be a
measurable set of positive measure. Our main objective
is to characterize the triples $(A,M,B)$ for which there
exists a constant $C>0$ such that
\begin{equation}
\label{eq:observabilityintro}
\|z(T)\|_{L^2(\Omega;\mathbb{R}^m)}^2
\leq C\int_{\mathcal{D}}|B^\top z(t,x)|^2\,dx\,dt,
\end{equation}
for every initial datum
$z_0\in L^2(\Omega;\mathbb{R}^m)$.
This inequality expresses that the energy of the full
system at time $T$ can be estimated from partial
observations on an arbitrary space-time measurable set
of positive measure.

A natural starting point is the finite-dimensional
counterpart
\begin{equation}
\label{eq:ODEintro}
z'(t)=M^\top z(t),
\qquad z(0)=z_0\in\mathbb{R}^m.
\end{equation}
Given any measurable set $E\subset(0,T)$ of positive
measure, the corresponding observability inequality is
\begin{equation}
\label{eq:observabilityintrosystem}
|z(T)|^2
\leq C\int_E|B^\top z(t)|^2\,dt,
\qquad \forall z_0\in\mathbb{R}^m.
\end{equation}
It is a classical consequence of the Kalman observability
criterion (see \cite[Chapter 1]{Coron2007}) that this inequality holds if and only if
\begin{equation}
\operatorname{rank}
\begin{bmatrix}
B & MB & \cdots & M^{m-1}B
\end{bmatrix}=m.
\end{equation}

For the parabolic system \eqref{eq:system}, this
finite-dimensional picture appears after projecting onto
the Dirichlet eigenmodes of $-\Delta$. More precisely, if
$(\lambda_k,e_k)$ is a Dirichlet eigenpair and
\[
z_k(t):=\int_\Omega z(t,x)e_k(x)\,dx,
\]
then
\begin{equation}
\label{eq:ODEmodeintro}
\begin{cases}
z_k'(t)=(-\lambda_k A^\top+M^\top)z_k(t),\\
y_k(t)=B^\top z_k(t),
\end{cases}
\end{equation}
is the corresponding finite-dimensional system together
with its observed output.
The Kalman criterion applied to this system leads directly
to the following frequency-dependent rank condition:
\begin{equation}\label{ass:H2}
\operatorname{rank}
\bigl[
B\quad (-\lambda_k A+M)B \quad
\dots \quad
(-\lambda_k A+M)^{m-1}B
\bigr]=m,
\qquad \forall k\geq 1.
\end{equation}
The purpose of this paper is to show that this condition
characterizes observability of the full parabolic system,
even when observations are restricted to an arbitrary
measurable subset of space-time.

The observability of coupled parabolic systems has been extensively studied over the past decades. We briefly review some key results, with particular emphasis on observability from measurable sets, to highlight the main motivation and novelty of the present work.

\medskip
\noindent\textbf{Scalar case.}
For the scalar heat equation, corresponding to $m=1$,
the coupling issue disappears. We may assume that
$A=a>0$ and $B=b\neq 0$.
The observability inequality
\eqref{eq:observabilityintro} was first established
for cylindrical observation sets
$\mathcal{D}=(0,T) \times \omega$, with
$\omega\subset\Omega$ a nonempty open subset, in the seminal
works \cite{LebeauRobbiano1995,FursikovImanuvilov1996}.
Several intermediate results retained a product structure in space and time.
Wang \cite{Wang2008} obtained bounded null controls supported on
$E\times\omega$, where $E\subset(0,T)$ has positive measure and
$\omega\subset\Omega$ is nonempty and open. Phung and Wang
\cite{PhungWang} established related observability estimates for parabolic
equations from such sets, using quantitative unique continuation at a
single time. Under analyticity assumptions, Apraiz and Escauriaza
\cite{ApraizEscauriaza2013} treated measurable spatial observation sets,
and Zhang \cite{Zhang2012} proved observability for the heat equation on
products $E\times\omega$ with both factors measurable and of positive
measure, for domains with real-analytic boundary.
The product structure was subsequently removed by Apraiz, Escauriaza,
Wang, and Zhang \cite{AEWZ}, who established observability from arbitrary
space-time measurable sets of positive measure.

A further direction concerns spatial sets of zero Lebesgue measure.
Burq and Moyano \cite{BurqMoyano2023} established spectral inequalities
and heat observability from spatial sets of positive $(d-\delta)$-dimensional
Hausdorff measure, for sufficiently small $\delta>0$, allowing elliptic
operators with Lipschitz coefficients.
Green, Le Balc'h, Martin, and Orsoni \cite{GreenLeBalchMartinOrsoni2025}
proved observability of the heat equation on bounded $C^1$ domains from
spatial sets of Hausdorff dimension strictly greater than $d-1$, with the
observation measured by a spatial supremum integrated in time.
Huang, G.~Wang, and M.~Wang \cite{HuangWangWang2026} subsequently refined
this geometric criterion using logarithmic-type Hausdorff contents.
Their observability estimates also cover certain observation sets of
Hausdorff dimension exactly $d-1$.

\medskip
\noindent\textbf{Weakly coupled case.}
A prototypical example of a weakly coupled parabolic
system is when $A= I_m$ and $M \in \R^{m \times m}$. Here, the coupling only occurs through the zeroth-order
term and does not affect the principal part of the
differential operator.

For cylindrical observation sets
$\mathcal{D}=(0,T)\times\omega$, with $\omega$ nonempty
and open, Ammar Khodja et al.
\cite{AmmarKhodjaEtAl2009Localized} established
observability under the classical Kalman rank condition
\begin{equation*}
\operatorname{rank}
\begin{bmatrix}
B & MB & \cdots & M^{m-1}B
\end{bmatrix}=m.
\end{equation*}
We refer to \cite{AmmarKhodjaEtAl2011Survey} for a survey
of controllability results for coupled parabolic systems
and of the role played by Kalman-type conditions.
Escauriaza, Montaner, and Zhang
\cite{EscauriazaMontanerZhang2015} subsequently
obtained observability from arbitrary space-time
measurable sets of positive measure for a class
of weakly coupled systems with two components.

First-order couplings have also been studied for controls supported on
$(0,T)\times\omega$, with $\omega$ nonempty and open. Duprez and Lissy
\cite{DuprezLissy2016} characterized null and approximate controllability
for a class of constant-coefficient systems with $m$ equations and $m-1$
controls, allowing both zeroth- and first-order couplings. Their approach
combines the fictitious control method, algebraic solvability, and Carleman
estimates. For two equations with variable coefficients, they obtained
further sufficient conditions and counterexamples in
\cite{DuprezLissy2018}. Extending our measurable-set result to first-order
couplings would require additional arguments: spatial derivatives do not,
in general, preserve the Dirichlet eigenspaces, so the independent modal
systems used here are no longer available.

\medskip
\noindent\textbf{Strongly coupled case.}
We now return to the system \eqref{eq:system}, where
the coupling occurs through the principal part.
For cylindrical observation sets
$\mathcal{D}=(0,T)\times\omega$, with $\omega$ nonempty and open, Lissy and Zuazua \cite{LissyZuazua2019}
established observability under the generalized Kalman rank
condition \eqref{ass:H2}.
Their argument can also be adapted to measurable spatial sets $\omega\subset\Omega$ of positive measure, by using the spectral inequality from measurable sets established in \cite{AEWZ}.

More recently, Fu, Wang, Yu, and Zhu \cite{Fu}
addressed the significantly more delicate question
of observability from arbitrary space-time measurable
sets. They considered the particular two-component
system \eqref{eq:system} with
\begin{equation*}
A^\top=
\begin{pmatrix}
a & -b\\
b & a
\end{pmatrix},
\qquad
B=
\begin{pmatrix}
1\\
0
\end{pmatrix},
\qquad M^\top=0,
\qquad a>0,\quad b\neq 0.
\end{equation*}

A major difficulty in this setting is that
high-frequency oscillations induced by the coupling
may lead to cancellations in the observed component.
Consequently, the pointwise-in-time interpolation
inequalities used to derive observability from
measurable sets in the scalar and weakly coupled
settings \cite{AEWZ,EscauriazaMontanerZhang2015}
may fail, as demonstrated in
\cite[Proposition 3.3]{Fu}.

To overcome this obstruction, the authors of
\cite{Fu} established an integral interpolation
inequality based on a Remez-type estimate, thereby
avoiding the need for pointwise-in-time interpolation
inequalities. Their proof relies on an explicit
representation of the solutions, made possible by
the particular structure of the coupling matrix.

A related quantitative problem was recently studied by Zhu \cite{ZhuVanishing2026}
for diffusion matrices $aI+\varepsilon K$, where $a>0$ and $K$ is
skew-symmetric, with observation on arbitrary measurable space-time sets.
Under a Kalman rank condition, that work determines sharp blow-up rates
of the observability cost as $\varepsilon\to0$, as well as the corresponding
minimum-norm control costs. At the level of individual modes,
\cite[Proposition 2.6]{ZhuVanishing2026} already uses the Tur\'an--Nazarov
inequality to pass to measurable time sets in arbitrary dimension.

Building on the integral interpolation approach of \cite{Fu}, we prove a
complete characterization under the generalized Kalman rank condition
\eqref{ass:H2} for general coercive diffusion matrices and arbitrary
zeroth-order coupling. This includes non-diagonalizable diffusion matrices
and rank conditions that depend on the spatial mode. The main quantitative
ingredient is a uniform finite-dimensional observability estimate with
polynomial dependence on the spatial frequency. It combines a consequence
of Nazarov's inequality for exponential polynomials with a quantitative
analysis of the mode-dependent Kalman matrix. Together with the Lebeau--Robbiano strategy, this yields
observability from arbitrary space-time measurable sets. 

\subsection{Main results}\label{sec-main-result}

We are now in a position to state our main result.

\begin{theorem}\label{theorem:obs}
Assume the parabolicity condition \eqref{ass:H1}.
Then the following assertions are equivalent.
\begin{itemize}
    \item The generalized Kalman rank condition \eqref{ass:H2} holds.
    \item For every $T>0$ and every measurable set
    $\mathcal{D}\subset(0,T)\times\Omega$ with
    $|\mathcal{D}|>0$, there exists a constant
    $C=C(A,M,B,T,\mathcal{D},\Omega)>0$ such that
    every solution of \eqref{eq:system} satisfies
    \begin{equation}\label{eq:observability}
        \|z(T)\|_{L^2(\Omega;\mathbb{R}^m)}
        \leq C\int_{\mathcal{D}}
        |B^\top z(t,x)|\,dx\,dt.
    \end{equation}
\end{itemize}
\end{theorem}

\begin{remark}
We make a few comments on Theorem \ref{theorem:obs}.
\begin{itemize}

\item Theorem~\ref{theorem:obs} extends the parabolic observability
result of Lissy and Zuazua \cite{LissyZuazua2019} from cylindrical sets
$(0,T)\times\omega$, with $\omega$ nonempty and open, to arbitrary
space-time measurable sets of positive measure, under the same generalized
Kalman rank condition. It also extends \cite[Theorem 1.1]{Fu} from their
two-component model to arbitrary coupling and observation matrices
satisfying the parabolicity and Kalman rank conditions.

\item The observability inequality
\eqref{eq:observability} is stronger than
the standard $L^2$-$L^2$ observability
inequality \eqref{eq:observabilityintro},
since the observation is measured in
$L^1(\mathcal{D})$ rather than $L^2(\mathcal{D})$.
We state the result in this form because
the $L^1$ estimate naturally arises in our proof.

\item The necessity of the generalized Kalman rank condition \eqref{ass:H2} for the system is straightforward. Indeed, assume that \eqref{ass:H2} fails for some index \(j_0\ge1\). Then there exists \(v\in\mathbb{R}^m\setminus\{0\}\) such that
\begin{align*}
B^\top (-\lambda_{j_0} A^\top+M^\top)^k v=0, \qquad k=0,\ldots,m-1.
\end{align*}
It follows from the Cayley--Hamilton theorem that $B^\top e^{t(-\lambda_{j_0} A^\top+M^\top)}v=0$ for every $t\geq 0$.
Let \(e_{j_0}\) be a nonzero Dirichlet eigenfunction of \(-\Delta\) with eigenvalue \(\lambda_{j_0}>0\).
Then 
\[
z(t,x)=e^{t(-\lambda_{j_0} A^\top+M^\top)}v\,e_{j_0}(x)
\]
solves \eqref{eq:system} and satisfies
\(B^\top z=0\) in \((0,T)\times\Omega\),
while \(z(T)\neq 0\).
This contradicts \eqref{eq:observability}.
Thus, the main contribution of the paper is the proof of sufficiency.

\end{itemize}
\end{remark}

By duality, Theorem~\ref{theorem:obs} yields null controllability with bounded controls supported on arbitrary measurable sets.
\begin{corollary}\label{cor:null-control}
Assume \eqref{ass:H1}. The Kalman rank condition \eqref{ass:H2} holds if and only if, for every $T>0$, every measurable set $\mathcal D\subset(0,T)\times\Omega$ of positive measure, and every $y_0\in L^2(\Omega;\mathbb R^m)$, there exists $u\in L^\infty(\mathcal D;\mathbb R^r)$ such that the solution of
\begin{equation}\label{eq:controlled-system}
\begin{cases}
\partial_t y-A\Delta y=My+\chi_{\mathcal D}Bu
&\text{in }(0,T)\times\Omega,\\
y=0&\text{on }(0,T)\times\partial\Omega,\\
y(0)=y_0&\text{in }\Omega
\end{cases}
\end{equation}
satisfies $y(T)=0$. Moreover, under \eqref{ass:H2}, the control can be chosen so that
\begin{equation}\label{eq:control-cost}
\|u\|_{L^\infty(\mathcal D;\mathbb R^r)}
\le C\|y_0\|_{L^2(\Omega;\mathbb R^m)},
\qquad C=C(A,M,B,T,\mathcal D,\Omega).
\end{equation}
\end{corollary}
\begin{proof}
Let $\varphi$ solve the backward adjoint equation
$-\partial_t\varphi-A^\top\Delta\varphi=M^\top\varphi$, with zero Dirichlet data and terminal datum $\varphi_T\in L^2(\Omega;\mathbb R^m)$.
Applying Theorem~\ref{theorem:obs} to $z(t)=\varphi(T-t)$ on the reflected set $\{(T-t,x):(t,x)\in\mathcal D\}$ gives
\[
\|\varphi(0)\|_{L^2(\Omega;\mathbb R^m)}
\le C\|B^\top\varphi\|_{L^1(\mathcal D;\mathbb R^r)}.
\]
Thus the functional $B^\top\varphi|_{\mathcal D}\mapsto-\langle y_0,\varphi(0)\rangle$ is well defined and bounded on a subspace of $L^1(\mathcal D;\mathbb R^r)$, with norm at most $C\|y_0\|_{L^2}$.
The Hahn--Banach theorem and $L^1$--$L^\infty$ duality give a representing control $u$ satisfying \eqref{eq:control-cost} and $$-\langle y_0,\varphi(0)\rangle = \int_{\mathcal D}u\cdot B^\top\varphi\,dx\,dt.$$
 The duality identity
\[
\langle y(T),\varphi_T\rangle
=\langle y_0,\varphi(0)\rangle
+\int_{\mathcal D}u\cdot B^\top\varphi\,dx\,dt=0
\]
then yields $y(T)=0$.
Conversely, if \eqref{ass:H2} fails at an eigenvalue $\lambda_k$, there is a nonzero $v$ such that $B^\top e^{t(-\lambda_kA^\top+M^\top)}v=0$ for all $t\ge0$. The corresponding backward mode $\varphi(t,x)=e^{(T-t)(-\lambda_kA^\top+M^\top)}v\,e_k(x)$ has $B^\top\varphi=0$ and $\varphi(0)\ne0$. Taking $y_0=\varphi(0)$ in the duality identity precludes $y(T)=0$.
\end{proof}

\begin{example}\label{ex:three-components}
Consider a system with three components and scalar observation, with
\[
A=\begin{pmatrix}1&0&0\\1&1&0\\0&1&1\end{pmatrix},\qquad
M=\begin{pmatrix}0&0&0\\0&1&0\\0&0&2\end{pmatrix},\qquad
B=\begin{pmatrix}1\\0\\0\end{pmatrix}.
\]
The matrix $A$ is non-diagonalizable, with a single Jordan block of size three, and satisfies
$\langle AX,X\rangle\ge(1-1/\sqrt2)|X|^2$ for all $X\in\mathbb R^3$.
For $D_\lambda=-\lambda A+M$, direct computation gives
\[
[B\ D_\lambda B\ D_\lambda^2B]
=\begin{pmatrix}
1&-\lambda&\lambda^2\\
0&-\lambda&2\lambda^2-\lambda\\
0&0&\lambda^2
\end{pmatrix},\qquad
\det[B\ D_\lambda B\ D_\lambda^2B]=-\lambda^3.
\]
Hence \eqref{ass:H2} holds at every Dirichlet eigenvalue. Theorem~\ref{theorem:obs}
therefore yields observability from any measurable space-time set of positive
measure by observing only the first component $B^\top z=z_1$.
This example lies outside both the two-component model of \cite{Fu}
and the diffusion family $aI+\varepsilon K$ with $K$ skew-symmetric
considered in \cite{ZhuVanishing2026}.
\end{example}

\subsection{Strategy of the proofs}\label{sec:strategy}

The main difficulty is to combine partial observation of the components
with observation on an arbitrary measurable subset of space-time.
As shown in \cite{Fu}, the oscillations generated by the principal coupling
can prevent pointwise-in-time interpolation estimates. We therefore work
with time-integrated observations. The new ingredient is a quantitative
estimate for general matrix evolutions, which treats non-diagonalizable
couplings and controls the dependence on the spatial frequency.

\emph{Uniform propagation from measurable time sets.}
After projection onto a Laplace eigenmode and rescaling time, the observed
quantity is $B^\top e^{-tG_j}\eta$, where
$G_j=A^\top-M^\top/\lambda_j$. Its scalar components are exponential
polynomials, with polynomial factors accounting for the Jordan blocks.
Proposition~\ref{prop-Remez} extends Nazarov's inequality to these functions
by approximating the polynomial factors with finite differences of
exponentials. The number of frequencies remains at most $m$, and the
constants are uniform in $j$. In particular, the estimate does not require
uniform bounds on Jordan bases or separation of the eigenvalues.
This replaces the special trigonometric representation used in \cite{Fu}
and permits general coupling matrices. For comparison, the Carleman approach
of Fern\'andez-Cara, Gonz\'alez-Burgos, and de Teresa
\cite[Theorem 1.1]{FernandezCara2015} for non-diagonalizable diffusion
matrices and space-time dependent zeroth-order coefficients requires
Jordan blocks of size at most four. In our constant-matrix setting,
the exponential-polynomial argument imposes no restriction on the size
of the Jordan blocks.

\emph{Direct quantitative observability of each mode.}
Quantitative estimates for fast controls in finite dimension provide an
important background for this step. Seidman \cite{Seidman1988} identified
the small-time growth of the minimum $L^2$ control cost in terms of the
Kalman structure; Seidman and Yong \cite{SeidmanYong1996} treated more
general $L^q$ control norms. Related constructive estimates based on rank
conditions can be found in \cite{Triggiani,Aze}. In the coupled parabolic
setting, Lissy and Zuazua \cite[Lemma 3.1]{LissyZuazua2019} already track
both the time horizon and the mode parameter in finite-dimensional
observability estimates used in their Lebeau--Robbiano argument.
Here too, a small-time estimate for a fixed matrix is insufficient:
we must control simultaneously the interval length $\ell$, the measurable
observation set $E$, and the spatial frequency $\lambda_j$. In particular,
the constant hidden in a fixed-matrix estimate cannot be allowed to depend
arbitrarily on $j$. The estimate below makes these dependencies explicit
for general coercive diffusion matrices and arbitrary zeroth-order coupling.

The central estimate is Proposition~\ref{thm:remez_kalman}. Its proof
separates recovery of the observation derivatives from recovery of the
state. Cayley--Hamilton gives a scalar differential equation, with uniformly
bounded coefficients, for each observed component. A Taylor expansion,
together with polynomial norm equivalence, then bounds the first $m$
observation derivatives by the time integral of the observation.
Using the weighted sum of derivatives
$\sum_{k=0}^{m-1}\ell^k|f^{(k)}(0)|$ makes the Taylor remainder absorbable
on an interval whose length is independent of the conditioning of the
Kalman matrix. The state is subsequently recovered from these derivatives
using the rank condition. Since the entries of the observability matrix
are polynomial in $1/\lambda_j$, a nonzero minor yields a loss bounded by
$(1+\lambda_j)^{m(m-1)}$. Combining the two estimates with the Remez
inequality gives
\[
|\eta|\le\frac{C(1+\lambda_j)^{m(m-1)}}{|E|^m}
\int_E|B^\top e^{-tG_j}\eta|\,dt,
\qquad E\subset(0,\ell),\quad0<\ell\le1.
\]
Thus all modes are covered by the same observability argument, with no
additional negative power of $\ell$. The polynomial dependence on frequency
is essential for the subsequent spectral argument.

\emph{From individual modes to the full system.}
Minkowski's inequality combines the mode estimates without a loss depending
on the number of retained modes. The spatial spectral inequality from
measurable sets then gives the low-frequency observation estimate.
Together with high-frequency parabolic dissipation, it yields the
integral interpolation inequality of Theorem~\ref{thm-interpolation}
by the Lebeau--Robbiano strategy. Finally, a density-point and telescoping
argument, following \cite{PhungWang,AEWZ,Fu}, gives
Theorem~\ref{theorem:obs}.
These final steps use established methods. The contribution lies in the
uniform matrix estimates that allow them to apply under the general
Kalman rank condition.

\section{Preliminaries}
In this section, we first introduce the basic notation used throughout the paper. We then recall several lemmas that are crucial for the proof of Theorem~\ref{theorem:obs} and can be found in the literature.

Let \(\mathbb{N}^+ := \{1, 2, \dots\}\) and \(\mathbb{R}^+ := (0, +\infty)\). For \(d \in \mathbb{N}^+\), \(r \in \mathbb{R}^+\), and \(x_0 \in \mathbb{R}^d\), we denote by \(B_r^d(x_0)\) the closed ball in \(\mathbb{R}^d\) centered at \(x_0\) with radius \(r\). In particular, we write \(B_r^d:=B_r^d(0)\). For any Lebesgue measurable set \(E \subset \mathbb{R}^d\), \(|E|\) denotes its Lebesgue measure. Given a matrix \(Q \in \mathbb{R}^{m \times n}\) with \(m, n \in \mathbb{N}^+\), we denote by \(Q^\top\) its transpose. For a Banach space \(X\), \(\|\cdot\|_X\) denotes its norm. In the finite-dimensional Euclidean space \(\mathbb{R}^d\), the norm \(\|\cdot\|_{\mathbb{R}^d}\) is the standard Euclidean norm; when no confusion can arise, we abbreviate it as \(|\cdot|\). If \(X\) is a Hilbert space, \(\langle \cdot, \cdot \rangle_X\) denotes its inner product.

Throughout the paper, the dependence of constants on the fixed domain
$\Omega$ is left implicit unless otherwise stated.

We next recall the following technical lemma, which is standard in the derivation of observability inequalities for the heat equation from space-time measurable sets; see \cite{AEWZ,Fu}.

\begin{lemma}\label{lemma-Fu-1}
Suppose that \(T > 0\) and \(B_R^d(x_0) \subset \Omega\) for some \(R > 0\) and \(x_0 \in \Omega\). Let \(\mathcal{D} \subset (0,T)\times B_R^d(x_0)\) be a measurable set with \(|\mathcal{D}| > 0\). Set
\begin{align}
    \mathcal{D}_t &:= \{x \in \Omega : (t, x) \in \mathcal{D}\} \text{ for } t \in (0, T), \\
    E &:= \left\{ t \in (0, T) : |\mathcal{D}_t| \ge \frac{|\mathcal{D}|}{2T} \right\}. 
\end{align}
Then, the following statements hold:
\begin{enumerate}
    \item[(i)] \(\mathcal{D}_t\) is a measurable subset of \(\Omega\) for a.e. \(t \in (0, T)\).
    
    \item[(ii)] For any \(f \in L^1((0,T)\times\Omega)\), the function \(t \mapsto \int_{\Omega} \chi_{\mathcal{D}_t}(x)|f(t, x)|dx\) is integrable over \((0, T)\).
    
    \item[(iii)] \(E\) is a measurable subset of \((0, T)\) and \(|E| \ge \frac{|\mathcal{D}|}{2|B_R^d|}\).
    
    \item[(iv)] \(\chi_E(t) \chi_{\mathcal{D}_t}(x) \le \chi_{\mathcal{D}}(t,x)\) in \((0,T)\times\Omega\).
\end{enumerate}
\end{lemma}

Let \(\{\lambda_j\}_{j\in\mathbb{N}^+}\), with \(0 < \lambda_1 < \lambda_2 \leq \cdots\) and \(\lambda_j \to +\infty\) as \(j \to +\infty\), be the eigenvalues of \(-\Delta\) under the homogeneous Dirichlet boundary condition. Let $\{e_j\}_{j\ge 1}$ be the corresponding Dirichlet eigenfunctions of $-\Delta$ on $\Omega$, constituting an orthonormal basis of $L^2(\Omega)$.

The following lemma will also be used to prove Theorem \ref{theorem:obs}. 

\begin{lemma}\label{lemma-Fu-2}
Suppose that \(T > 0\) and \(B_{4R}^d(x_0) \subset \Omega\) for some \(R > 0\) and \(x_0 \in \Omega\). Let \(\mathcal{D} \subset (0,T)\times B_R^d(x_0)\) be measurable with $|\mathcal{D}|>0$. Let $E$ and $\mathcal D_t$ be defined as in Lemma~\ref{lemma-Fu-1}. For every fixed $\Lambda_1>0$, there exists a constant $C(R,|\mathcal D|/T)>0$ such that, for almost every $t\in E$, every family $\{a_j\}\subset\mathbb R^r$, and every $\Lambda>\Lambda_1$,
\begin{equation}\label{eq-vector-spectral-inequality}
    \sum_{\lambda_j\le \Lambda}^{} |a_j|^2\le C(R, |\mathcal{D}|/T) e^{C(R, |\mathcal{D}|/T)\sqrt{\Lambda}} \left( \left\| \  \chi_{\mathcal{D}_t} \sum_{\lambda_j\le \Lambda}^{} a_j e_j \right\|_{L^1(\Omega; \mathbb{R}^r)} \right)^2  .
\end{equation}
\end{lemma}

\begin{proof}
Fix a time $t\in E$ for which the slice is measurable and write $a_j=(a_j^1,\ldots,a_j^r)$. Set
\[
f_k=\sum_{\lambda_j\le\Lambda}a_j^k e_j,
\qquad F=(f_1,\ldots,f_r)=\sum_{\lambda_j\le\Lambda}a_j e_j.
\]
Applying the scalar spectral inequality from \cite[Corollary 2.6]{Fu}
(see also \cite{AEWZ}) to each $f_k$ and summing over $k$, we obtain
\begin{equation}\label{eq-264}
\sum_{\lambda_j\le\Lambda}|a_j|^2
=\sum_{k=1}^r\sum_{\lambda_j\le\Lambda}|a_j^k|^2
\le Ce^{C\sqrt\Lambda}\sum_{k=1}^r
\left(\int_{\mathcal D_t}|f_k(x)|\,dx\right)^2,
\end{equation}
where $C=C(R,|\mathcal D|/T)$ is independent of $k$ and $r$.
The triangle inequality for the integral in $\mathbb R^r$ gives
\begin{align*}
\left[\sum_{k=1}^r\left(\int_{\mathcal D_t}|f_k(x)|\,dx\right)^2\right]^{1/2}
&=\left|\int_{\mathcal D_t}(|f_1(x)|,\ldots,|f_r(x)|)\,dx\right|_{\mathbb R^r}\\
&\le\int_{\mathcal D_t}\left(\sum_{k=1}^r|f_k(x)|^2\right)^{1/2}\,dx\\
&=\|\chi_{\mathcal D_t}F\|_{L^1(\Omega;\mathbb R^r)}.
\end{align*}
Squaring this estimate and inserting it into \eqref{eq-264}
yields \eqref{eq-vector-spectral-inequality}.
\end{proof}

\section{Integral interpolation inequality}\label{sec2}

This section is devoted to proving the following integral-type interpolation observability inequality for \eqref{eq:system}, which plays a crucial role in the proof of Theorem \ref{theorem:obs}.

We write $S(t)=e^{t(A^\top\Delta+M^\top)}$ for the solution semigroup
with homogeneous Dirichlet boundary conditions. In this section, for an
initial datum $z\in L^2(\Omega;\mathbb R^m)$, we use the convention
$z(t)=S(t)z$.

We fix $\Lambda_1\ge\max\{1,2\|M\|/\alpha\}$ and use the observation
sets constructed in Section~2. More precisely, let
$\mathcal D\subset (0,T)\times B_R^d(x_0)$ have positive measure, with
$B_{4R}^d(x_0)\subset\Omega$, and set
\[
\mathcal D_t=\{x\in\Omega:(t,x)\in\mathcal D\},\qquad
E=\left\{t\in(0,T):|\mathcal D_t|\ge\frac{|\mathcal D|}{2T}\right\}.
\]
By Lemma~\ref{lemma-Fu-1}, these slices are measurable for almost every $t$,
the corresponding observation integrals are measurable, and $|E|>0$.
Lemma~\ref{lemma-Fu-2} provides a constant $D=D(R,|\mathcal D|/T)>0$
such that, for almost every $t\in E$, every $\Lambda>\Lambda_1$,
and every family $\{a_j\}\subset\mathbb R^r$,
\begin{equation}\label{eq-assumption-Lebeau}
\sum_{\lambda_j\le\Lambda}|a_j|^2
\le D e^{D\sqrt\Lambda}
\left\|\chi_{\mathcal D_t}\sum_{\lambda_j\le\Lambda}a_je_j
\right\|_{L^1(\Omega;\mathbb R^r)}^2.
\end{equation}
These properties are preserved when $E$ is replaced by a measurable subset
of positive measure, or when the time interval and the observation pair
$(E,\mathcal D_t)$ are translated together. In both cases, the spectral
constant $D$ is inherited from the original pair; it is not recomputed
from the restricted set. The result below applies to these pairs as well.
All constants may depend on the fixed domain $\Omega$.

\begin{theorem}\label{thm-interpolation}
Assume \eqref{ass:H1} and \eqref{ass:H2}, and let $(E,\mathcal D_t)$ be
an observation pair described above.
There exists a constant $C_8:= C_8(A,M,B,\Lambda_1,m,D,\Omega)>0$ such that, for every $\theta\in (0,1)$, every $0<S_1<S_2\le T$ satisfying $|E\cap [S_1,S_2]|>0$, and every $z\in L^2(\Omega;\mathbb{R}^m)$,
\begin{multline}\label{eq-interpolation-inequality}
\|z(S_2)\|_{L^2(\Omega;\mathbb{R}^m)}\\\leq
\tilde C\left(\frac{1}{|E\cap[S_1,S_2]|}\int_{S_1}^{S_2}
\chi_E(t)\|\chi_{\mathcal D_t}B^\top z(t)\|_{L^1(\Omega;\mathbb{R}^r)}\,dt\right)^{1-\theta}
\|z\|_{L^2(\Omega;\mathbb{R}^m)}^\theta.
\end{multline}
where 
$$
    \tilde{C}=\frac{C_8e^{C_8\left ( \frac{S_2}{1-\theta}+\frac{1}{\theta S_1} \right )}}{|E \cap [S_1, S_2]|^{m-1}}.
$$
\end{theorem}

We split $z(t)=P_\Lambda z(t)+(I-P_\Lambda)z(t)$ into low and high frequencies, where $P_\Lambda$ is the spectral projection
\begin{align}
    P_\Lambda u = \sum_{\lambda_j\le \Lambda}^{} \left \langle u,e_j \right \rangle e_j,\quad u\in L^2(\Omega).
\end{align}
The projection $P_\Lambda$ acts componentwise on vector-valued functions
and commutes with constant matrices.

In Section \ref{sec3.1}, we prove the low-frequency estimate \eqref{eq-final-low} for $\left \| P_\Lambda z(t) \right \|_{L^2(\Omega;\mathbb{R}^m) }$. To achieve this, we first establish in Proposition \ref{prop-Remez} a Remez--type inequality for functions of the form
$f(t) = \sum_{\nu=1}^s e^{-\mu_\nu t} \mathcal{P}_\nu(t)$,
where \(\mathcal{P}_\nu\) is a polynomial and \(\mu_1, \dots, \mu_s\) are the distinct eigenvalues of \(A-\frac{M}{\lambda_j}\in \mathbb{R}^{m\times m}\). Based on Proposition \ref{prop-Remez} and the generalized Kalman rank condition \eqref{ass:H2}, we prove the unified observability estimate in Proposition~\ref{thm:remez_kalman}, which is the key ingredient in the proof of \eqref{eq-final-low}.

In Section \ref{sec3.2}, we combine the low-frequency estimate \eqref{eq-final-low} with the Lebeau--Robbiano strategy to prove Theorem \ref{thm-interpolation}.

\subsection{Remez--type inequality and low-frequency estimate}\,\label{sec3.1}

In \cite{Fu}, the authors considered \eqref{eq:system} with
\begin{equation*}
    A^\top := \begin{pmatrix}
    a & -b \\
    b & a
    \end{pmatrix},
    \qquad
    B^\top = (1,0),
    \qquad
    M^\top =0,
\end{equation*}
where \(a>0\) and \(b\neq 0\). They defined
\begin{equation*}
    \mathcal{A} := 
        \begin{pmatrix}
        a\Delta & -b\Delta \\
        b\Delta & a\Delta
        \end{pmatrix}.
\end{equation*}
For each \(z \in L^2(\Omega; \mathbb{R}^2)\), let
\(\{z_j\}_{j \in \mathbb{N}^+} := \{(z_{1,j}, z_{2,j})^\top\}_{j \in \mathbb{N}^+} \subset \mathbb{R}^2\)
denote the Fourier coefficients of \(z\) with respect to the orthonormal basis
\(\{e_j\}_{j \in \mathbb{N}^+}\). Then
\begin{equation}\label{eq-solution-Fu}
    e^{\mathcal{A}t} z 
    = \sum_{j=1}^{+\infty} e^{-A^\top\lambda_j t} z_j e_j 
    = \sum_{j=1}^{+\infty} e^{-a\lambda_j t} 
    \begin{pmatrix}
    \cos(\lambda_j b t) & \sin(\lambda_j b t) \\
    -\sin(\lambda_j b t) & \cos(\lambda_j b t)
    \end{pmatrix}
    z_j e_j.
\end{equation}
Observing this representation, the authors established a Remez-type inequality for
\begin{multline*}
\mathcal{T}_n := \Bigl\{ f \in C([-\pi, \pi]; \mathbb{R}) : \\
    f(\theta) = \sum_{k=0}^n \big( a_k \sin(k\theta) + b_k \cos(k\theta) \big), 
 \theta \in (-\pi, \pi), \{(a_k, b_k)^\top\}_{k=0}^n \subset \mathbb{R}^2 \Bigr\}.
\end{multline*}
and used it to prove the low-frequency estimate. However, for general matrices $A,M\in\mathbb{R}^{m\times m}$ and $B\in\mathbb{R}^{m\times r}$, simple solution formulas such as \eqref{eq-solution-Fu} are no longer available. For each $z\in L^2(\Omega;\mathbb{R}^m)$, the solution of \eqref{eq:system} has the expansion
\begin{align}\label{eq:expansion}
z(t,x) = e^{(A^\top\Delta + M^\top) t}z = \sum_{j=1}^{+\infty} e^{-(A^\top - \frac{M^\top}{\lambda_j})\lambda_j t} z_j e_j.
\end{align}
where 
\begin{align}
    \{z_j\}_{j \in \mathbb{N}^+} := \{(z_{1,j}, z_{2,j},\ldots, z_{m,j})^\top\}_{j \in \mathbb{N}^+} \subset \mathbb{R}^m,
\end{align}
denote the Fourier coefficients of $z$ with respect to the orthonormal basis $\{e_j\}_{j \in \mathbb{N}^+}$.

We first analyze the form of each component of
$B^\top e^{-(A^\top - \frac{M^\top}{\lambda_j}) t} \eta$
for $\eta\in \mathbb{R}^m$, and then establish the corresponding Remez-type inequality.

Set
\begin{align*}
G_j:=A^\top-\frac{M^\top}{\lambda_j}.
\end{align*}
For all $j\ge1$, we have the uniform bound
\begin{equation}\label{eq-spectrum-uniform-bound}
\max_{\mu\in\sigma(G_j)}|\operatorname{Re}\mu|
\le\|G_j\|\le\|A\|+\frac{\|M\|}{\lambda_1}=:K_G.
\end{equation}
\begin{proposition}\label{prop-Remez}
Let \(\mu_1, \dots, \mu_s\) be the distinct eigenvalues of $G_j\in\mathbb R^{m\times m}$, with algebraic multiplicities 
\(d_1, \dots, d_s\), respectively, so that \(\sum_{\nu=1}^s d_\nu = m\). For $\eta\in\mathbb{R}^m$, the following statements hold.

(i) Any scalar component of 
$B^\top e^{-tG_j} \eta$
can be written as
\begin{align}\label{eq:gen_exp_rep}
    f(t) = \sum_{\nu=1}^s e^{-\mu_\nu t} \mathcal{P}_\nu(t), 
    \qquad 
    \sum_{\nu=1}^s (\deg \mathcal{P}_\nu + 1) \le m,
\end{align}
where each $\mathcal P_\nu$ is a polynomial with complex coefficients; zero polynomials are omitted from the sum of degrees. 

(ii) Let $j\ge1$, let
$I\subset\mathbb{R}$ be an interval with $0<|I|\le1$,
and let $E\subset I$ be a measurable subset with
$|E|>0$. Then there exists a constant $C=C(A,M,m,\lambda_1,r)>0$,
independent of $j$, such that
\begin{align}\label{eq:nazarov_extended}
    \sup_I |f| 
    \le C(A,M,m,\lambda_1,r) 
    \left( \frac{|I|}{|E|} \right)^{m-1} \sup_E |f|. 
\end{align}

(iii) Under the same assumptions,
\begin{align}\label{eq:BA}
    \int_I |B^\top e^{-tG_j} \eta|dt \le C(A,M,m,\lambda_1,r) \left(\frac{|I|}{|E|}\right)^m \int_E |B^\top e^{-tG_j} \eta|dt.
\end{align}
\end{proposition}

\begin{proof} 
(i) By the Jordan canonical form, there exists an invertible matrix \(P\in \mathbb{C}^{m\times m}\) such that \(G_j = P J P^{-1}\), where 
$J=\operatorname{diag}(J_1,\dots,J_s)$, where $J_\nu$ groups all Jordan blocks associated with $\mu_\nu$. Thus $J_\nu$ has size $d_\nu$ and satisfies
\begin{align}\label{eq:jordan_block}
    J_\nu = \mu_\nu I_{d_\nu} + N_\nu,\ N_\nu^{d_\nu} = 0, \qquad 1\le \nu\le s,
\end{align}
Here $N_\nu$ is the direct sum of the corresponding nilpotent Jordan blocks; it need not be a single shift matrix.

By the Jordan decomposition,
\begin{align*}
    e^{-tG_j} = P e^{-J t} P^{-1} = P \operatorname{diag}(e^{-J_1t}, \dots, e^{-J_st}) P^{-1},
\end{align*}
where
\begin{align}\label{eq:exp_jordan}
    e^{-J_\nu t} = e^{-\mu_\nu t} \sum_{k=0}^{d_\nu-1} \frac{(-t)^k}{k!} N_\nu^k.
\end{align}
The expansion of the exponential and the identity
\(N_\nu^{d_\nu}=0\) show that each component of
\(e^{-J_\nu t}\) is of the form \(e^{-\mu_\nu t} P_\nu(t)\),
with \(\deg P_\nu \le d_\nu-1\). Hence, any scalar component of
$B^\top e^{-tG_j} \eta$
can be written as
\begin{align*}
    f(t) = \sum_{\nu=1}^s e^{-\mu_\nu t} \mathcal{P}_\nu(t)
    =\sum_{\nu=1}^s \sum_{k=0}^{d_\nu-1} a_{\nu,k} t^k e^{-\mu_\nu t},
\end{align*}
where $a_{\nu,k}\in\mathbb C$. Omitting zero polynomials from the sum of degrees, we obtain
\begin{align*}
    \sum_{\nu=1}^s (\deg \mathcal{P}_\nu + 1) \le m.
\end{align*}

(ii) Our strategy is as follows: we construct an exponential polynomial $f_\varepsilon$ of order $N$ such that
\begin{align*}
    f_\varepsilon \to f\ \  \text{ uniformly on } I.
\end{align*}
We apply Nazarov's inequality from \cite{Nazarov} to $f_\varepsilon$ and then pass to $f$ by an approximation argument.
\begin{lemma}[Nazarov's inequality]
Let
\[
    p(t)=\sum_{j=1}^{N} c_j e^{\alpha_j t}
\]
be an exponential polynomial of order $N$, where $\alpha_j\in\mathbb{C}$ and $c_j\in\mathbb{C}$.
Let $I\subset\mathbb{R}$ be an interval, and let $E\subset I$ be a measurable subset with positive Lebesgue measure.
Then there exists a constant $C>0$, depending only on $N$, such that
\begin{equation}\label{eq:nazarov_turan}
    \sup_{I}|p|\le
    \exp\Big(|I|\max_{j}|\operatorname{Re}\alpha_j|\Big)
    \left(\frac{C|I|}{|E|}\right)^{N-1}
\sup_{E}|p|.
\end{equation}
\end{lemma} 
 
For each monomial \(t^k e^{-\mu_\nu t}\) 
appearing in \eqref{eq:gen_exp_rep}, with \(0 \le k \le d_\nu-1\), 
define the following exponential polynomial
\begin{align}\label{eq-diff}
    p_{\nu,k,\varepsilon}(t) 
    := \varepsilon^{-k} (e^{\varepsilon t} - 1)^k       e^{-\mu_\nu t}
    = \sum_{r=0}^k \binom{k}{r} (-1)^{k-r}              \varepsilon^{-k} 
        e^{(-\mu_\nu + r \varepsilon) t}. 
\end{align}
Since $(e^{\varepsilon t}-1)/\varepsilon\to t$ uniformly on the fixed interval $I$,
\begin{equation}\label{eq:uniform_conv}
\sup_{t\in I}|p_{\nu,k,\varepsilon}(t)-t^k e^{-\mu_\nu t}|\longrightarrow0
\qquad\text{as }\varepsilon\to0^+.
\end{equation}

For each \(\varepsilon>0\), define the pure exponential sum
\begin{align}
    f_\varepsilon(t):=\sum_{\nu=1}^s \sum_{k=0}^{d_\nu-1} a_{\nu,k}\, p_{\nu,k,\varepsilon}(t). \label{eq:f_epsilon_def}
\end{align}
Substituting \eqref{eq-diff} into \eqref{eq:f_epsilon_def} yields
\[
    f_\varepsilon(t)=
    \sum_{\nu=1}^s \sum_{k=0}^{d_\nu-1} \sum_{r=0}^k
    a_{\nu,k} \binom{k}{r}(-1)^{k-r}\varepsilon^{-k}
    e^{(-\mu_\nu+r\varepsilon)t},
\]
which is an exponential polynomial with at most \(m\) distinct frequencies. Indeed, for each fixed \(\nu\), the index \(r\) ranges from \(0\) to \(k\), and since \(k\le d_\nu-1\), we have \(0\le r\le d_\nu-1\). Hence, after collecting terms with identical frequencies, the total number \(N_\varepsilon\) of distinct frequencies in \(f_\varepsilon\) satisfies
\[
    N_\varepsilon \le \sum_{\nu=1}^s d_\nu = m.
\]
Moreover, the uniform convergence \(p_{\nu,k,\varepsilon}\to t^k e^{-\mu_\nu t}\) on $I$ as \(\varepsilon\to0^+\) (cf.~\eqref{eq:uniform_conv}) implies
\[
    f_\varepsilon \longrightarrow f \quad \text{uniformly on } I.
\]

For \(0 < \varepsilon \le 1\), the real parts of the frequencies are uniformly bounded by 
\(\max_\nu |\operatorname{Re} (\mu_\nu)| + m\). 
Applying \eqref{eq:nazarov_turan} to \(f_\varepsilon\) yields
\begin{align*}
    \sup_I |f_\varepsilon| 
    \le \exp\!\left(|I| (\max_\nu |\operatorname{Re} (\mu_\nu)|+m)\right) 
    \left( \frac{C(m)|I|}{|E|} \right)^{m-1} \sup_E |f_\varepsilon|.
\end{align*}
For every $j\ge1$,
\eqref{eq-spectrum-uniform-bound} and $|I|\le1$ show
that the exponential factor in the preceding inequality
is bounded uniformly with respect to $j$. Thus, taking the limit
\(\varepsilon \to 0^+\) yields the generalized estimate
\begin{align*}
\sup_I |f| 
    \le C(A,M,m,\lambda_1,r) 
    \left( \frac{|I|}{|E|} \right)^{m-1} \sup_E |f|.
\end{align*}

(iii) First, by \eqref{eq:nazarov_extended}, we prove 
\begin{align}\label{eq-367}
    \int_I |f(t)|dt \le C(A,M,m,\lambda_1,r) \left(\frac{|I|}{|E|}\right)^m \int_E |f(t)|dt,
\end{align}
where $f(t)$ is any scalar component of
$B^\top e^{-tG_j} \eta$. Set
\begin{align}
\label{eq:defME}
    M_E := \frac{1}{|E|}\int_E |f(t)|\,dt,\qquad 
    E_0 := \{t\in E : |f(t)| \le 2M_E\}.
\end{align}
Then $|E_0| \ge |E|/2$. Applying \eqref{eq:nazarov_extended} to $E_0$ yields
\begin{align*}
    \sup_I |f| \le C(A,M,m,\lambda_1,r) 
    \left( \frac{|I|}{|E_0|} \right)^{m-1} \sup_{E_0} |f|,
\end{align*}
Therefore,
\begin{align*}
    \int_I |f(t)|\,dt \le C(A,M,m,\lambda_1,r) \left( \frac{|I|}{|E_0|} \right)^{m-1}\cdot|I|\cdot \sup_{E_0} |f|,
\end{align*}
which implies \eqref{eq-367} by using the definition \eqref{eq:defME}.

Write
\begin{align*}
    B^\top e^{-tG_j} \eta
    =\bigl(f_1(t),f_2(t),\ldots,f_r(t)\bigr)\in \mathbb{R}^r.
\end{align*}
Applying \eqref{eq-367} to each \(f_i\), $i=1,\ldots,r$, and summing over $i$, we obtain
\begin{align}\label{eq-393}
    \int_I |f_1(t)|+\cdots+|f_r(t)| dt \le C(A,M,m,\lambda_1,r) \left(\frac{|I|}{|E|}\right)^m \int_E |f_1(t)|+\cdots+|f_r(t)|dt.
\end{align}
Observe that
\begin{align}\label{eq-397}
    |f_1(t)|+\cdots+|f_r(t)|\asymp \sqrt{|f_1(t)|^2+\cdots+|f_r(t)|^2}=|B^\top e^{-tG_j} \eta|_{\mathbb{R}^r}.
\end{align}
Here, $f\asymp g$ means that there exists a constant $C>0$ such that $f\le Cg$ and $g\le Cf$. Then \eqref{eq:BA} follows from \eqref{eq-393} and \eqref{eq-397}.
\end{proof}

We next establish a direct observability estimate valid for every mode.
Throughout the remainder of this section, constants may depend on the fixed
matrices, their dimensions, and $\Omega$, but not on the mode index.
Set $p=m(m-1)$ and $\omega=\|M\|$.

\begin{proposition}\label{thm:remez_kalman}
Assume \eqref{ass:H1} and \eqref{ass:H2}. The following estimates hold for every $j\ge1$, with a constant $C=C(A,M,B,m,\Omega)>0$.
For $0<\ell\le1$, every measurable $E\subset(0,\ell)$ of positive measure,
and every $\eta\in\mathbb R^m$,
\begin{equation}\label{eq-short-time}
|\eta|\le \frac{C(1+\lambda_j)^p}{|E|^m}
\int_E|B^\top e^{-tG_j}\eta|\,dt.
\end{equation}
For every $L>0$, every measurable $F\subset(0,L)$ of positive measure,
and every $\xi\in\mathbb R^m$,
\begin{equation}\label{eq:remez_kalman}
|e^{-LG_j}\xi|\le C(1+\lambda_j)^p e^{\omega L/\lambda_j}
\left(\frac{1+L}{|F|}\right)^m
\int_F|B^\top e^{-tG_j}\xi|\,dt.
\end{equation}
If $\lambda_j\alpha\ge\omega$, the exponential factor in
\eqref{eq:remez_kalman} can be omitted.
\end{proposition}

\begin{proof}
\emph{Step 1: Uniform recovery of the observation derivatives.}
We claim that, for $f(t)=B^\top e^{-tG_j}\eta$ and $0<\ell\le1$,
\begin{equation}\label{eq:jet-observability}
\sum_{k=0}^{m-1}|f^{(k)}(0)|\le C\ell^{-m}\int_0^\ell|f(t)|\,dt,
\end{equation}
where $C$ depends only on $m$ and the uniform bound on $\|G_j\|$.
Let $K=\|A\|+\|M\|/\lambda_1$, so that $\|G_j\|\le K$.
Write the characteristic polynomial of $-G_j$ as
\[
\det(\zeta I_m+G_j)=\zeta^m+a_{m-1,j}\zeta^{m-1}+\cdots+a_{0,j}.
\]
Its coefficients are uniformly bounded in terms of $m$ and $K$.
Since $f^{(k)}(t)=B^\top(-G_j)^k e^{-tG_j}\eta$,
Cayley--Hamilton gives
\[
f^{(m)}+a_{m-1,j}f^{(m-1)}+\cdots+a_{0,j}f=0.
\]
Thus $U=(f,f',\ldots,f^{(m-1)})\in(\mathbb R^r)^m$ satisfies
$U'=(\mathcal C_j\otimes I_r)U$, where
\[
\mathcal C_j=
\begin{pmatrix}
0&1&0&\cdots&0\\
0&0&1&\cdots&0\\
\vdots&&\ddots&\ddots&\vdots\\
0&\cdots&0&0&1\\
-a_{0,j}&-a_{1,j}&\cdots&-a_{m-2,j}&-a_{m-1,j}
\end{pmatrix}.
\]
For $m=1$, this matrix is simply $(-a_{0,j})$.
Equip $(\mathbb R^r)^m$ with the norm
\[\|(u_0,\ldots,u_{m-1})\|_*=\sum_{k=0}^{m-1}|u_k|.\]
The operator norms of $\mathcal C_j\otimes I_r$ are bounded by a constant
$K_0=K_0(m,K)$, independently of $j$ and $r$. Hence
\[
\|U(t)\|_*\le e^{K_0t}\|U(0)\|_*,\qquad 0\le t\le1.
\]
Using the last equation of this system, we conclude that
\[
\sup_{0\le t\le1}|f^{(m)}(t)|
\le C_0\sum_{k=0}^{m-1}|f^{(k)}(0)|,
\qquad C_0=C_0(m,K).
\]
We decompose $f=P+R$, where $P$ is its Taylor polynomial of degree
at most $m-1$ at zero and $R$ is the corresponding integral remainder:
\[
P(t)=\sum_{k=0}^{m-1}\frac{t^k}{k!}f^{(k)}(0),\qquad
R(t)=\frac1{(m-1)!}\int_0^t(t-s)^{m-1}f^{(m)}(s)\,ds.
\]
For $0<\ell\le1$, we also introduce the weighted sum
\[
J_\ell:=\sum_{k=0}^{m-1}\ell^k|f^{(k)}(0)|.
\]
We first justify the polynomial estimate
\[
\int_0^\ell|P(t)|\,dt\ge c_0\ell J_\ell,
\qquad c_0=c_0(m)>0.
\]
On the scalar polynomial space of degree at most $m-1$, the map
$q\mapsto\int_0^1|q(s)|\,ds$ is a norm. Each derivative evaluation at zero
is a continuous linear functional on this finite-dimensional space, so
\[
|q^{(k)}(0)|\le C_m\int_0^1|q(s)|\,ds,
\qquad 0\le k\le m-1.
\]
Apply this to the vector polynomial $Q(s)=P(\ell s)$ as follows.
For each $k$ with $f^{(k)}(0)\ne0$, choose the unit vector
$v_k=f^{(k)}(0)/|f^{(k)}(0)|$ and use the scalar polynomial
$q_k(s)=v_k\cdot Q(s)$. Since
$q_k^{(k)}(0)=\ell^k|f^{(k)}(0)|$ and $|q_k(s)|\le|Q(s)|$, we get
\[
\ell^k|f^{(k)}(0)|\le C_m\int_0^1|Q(s)|\,ds.
\]
The same inequality is immediate for a zero derivative. Summing over $k$
and changing variables gives
$J_\ell\le mC_m\ell^{-1}\int_0^\ell|P(t)|\,dt$, as claimed.

We next estimate the remainder. Using its integral representation and
the bound on $f^{(m)}$ obtained above, we have, for $0\le t\le\ell$,
\begin{align*}
|R(t)|
&\le\frac{C_0}{(m-1)!}
\left(\sum_{k=0}^{m-1}|f^{(k)}(0)|\right)
\int_0^t(t-s)^{m-1}\,ds\\
&=\frac{C_0t^m}{m!}\sum_{k=0}^{m-1}|f^{(k)}(0)|.
\end{align*}
Integrating in $t$ therefore gives
\[
\int_0^\ell|R(t)|\,dt
\le\frac{C_0\ell^{m+1}}{m!(m+1)}
\sum_{k=0}^{m-1}|f^{(k)}(0)|.
\]
To express this bound in terms of $J_\ell$, note that
$0<\ell\le1$ and $k\le m-1$ imply $\ell^k\ge\ell^{m-1}$. Thus
\[
J_\ell=\sum_{k=0}^{m-1}\ell^k|f^{(k)}(0)|
\ge\ell^{m-1}\sum_{k=0}^{m-1}|f^{(k)}(0)|.
\]
It follows, with $C_R=C_0/(m!(m+1))$, that
\[
\int_0^\ell|R(t)|\,dt
\le C_R\ell^{m+1}\ell^{-(m-1)}J_\ell
=C_R\ell^2J_\ell.
\]
Choose
\[
\ell_0=\min\left\{1,\frac{c_0}{2(C_R+1)}\right\}.
\]
For $0<\ell\le\ell_0$, the triangle inequality gives
\[
\int_0^\ell|f(t)|\,dt
\ge\int_0^\ell|P(t)|\,dt-\int_0^\ell|R(t)|\,dt
\ge\frac{c_0}{2}\ell J_\ell.
\]
Combining this with
$\sum_{k=0}^{m-1}|f^{(k)}(0)|\le\ell^{-(m-1)}J_\ell$ proves
\eqref{eq:jet-observability} for $\ell\le\ell_0$.
For $\ell_0<\ell\le1$, apply this estimate on $(0,\ell_0)$:
\[
\sum_{k=0}^{m-1}|f^{(k)}(0)|
\le\frac{2}{c_0\ell_0^m}\int_0^{\ell_0}|f(t)|\,dt
\le\frac{2}{c_0\ell_0^m}\ell^{-m}\int_0^\ell|f(t)|\,dt.
\]
This proves \eqref{eq:jet-observability} on the whole range $0<\ell\le1$.
In particular, $\ell_0$ depends only on $m$ and $K$, and not on
$B$ or the conditioning of the Kalman matrix.

\emph{Step 2: Quantitative recovery of the state.}
Consider the polynomial matrix
\[
\mathcal O(s)=\begin{pmatrix}
B^\top\\ B^\top(A^\top-sM^\top)\\ \vdots\\
B^\top(A^\top-sM^\top)^{m-1}
\end{pmatrix}.
\]
By the Kalman rank condition \eqref{ass:H2}, $\mathcal O(1/\lambda_j)$ has rank $m$ for every $j$.

Fix once and for all an $m\times m$ minor $S(s)$ of $\mathcal O(s)$
whose determinant $d(s)$ is not identically zero. Such a minor exists
because $\mathcal O(1/\lambda_j)$ has rank $m$. This choice is fixed
throughout the argument, independently of the mode index $j$; hence the
coefficients of $S$ and $d$ depend only on the fixed matrices $A,M,B$.
Each row has degree at most $m-1$, and the determinant is a sum of
products of $m$ entries; hence $\deg d\le m(m-1)=p$.
Factor out its zero at the origin:
\[
d(s)=s^q h(s),\qquad 0\le q\le p,\qquad h(0)\ne0.
\]
The polynomial $h$ is also fixed, rather than chosen separately for each
$j$. By continuity, we may choose $s_0\in(0,1]$, depending only on $h$,
such that $|h(s)|\ge|h(0)|/2$ for $0\le s\le s_0$.
Set
\[
C_0=\frac{2}{|h(0)|}\max_{0\le s\le s_0}
\|\operatorname{adj}S(s)\|<\infty.
\]
The cofactor formula gives
\[
\|S(s)^{-1}\|
=\left\|\frac{\operatorname{adj}S(s)}{s^q h(s)}\right\|
\le C_0s^{-q},\qquad 0<s\le s_0.
\]
Thus $s_0$ and $C_0$ depend on the fixed matrices $A,M,B$ through the
chosen minor, but neither depends on $j$ nor on the evaluation point
$s=1/\lambda_j$.
It follows that, for $\lambda_j\ge s_0^{-1}$,
\begin{align*}
|\eta| = |S(1/\lambda_j)^{-1}S(1/\lambda_j) \eta|&\le C\lambda_j^q|\mathcal O(1/\lambda_j)\eta|\\
&\le C(1+\lambda_j)^p\sum_{k=0}^{m-1}|B^\top G_j^k\eta|.
\end{align*}
For the finitely many indices with $\lambda_j<s_0^{-1}$, define
\[
\kappa_j=\min_{|\eta|=1}|\mathcal O(1/\lambda_j)\eta|>0.
\]
Positivity follows from injectivity and compactness of the unit sphere.
If this finite set is nonempty, its minimum $\kappa_*>0$ satisfies
$|\eta|\le\kappa_*^{-1}|\mathcal O(1/\lambda_j)\eta|$ for each of these modes.
Enlarging $C$ thus gives
\begin{equation}\label{eq:quantitative-kalman}
|\eta|\le C(1+\lambda_j)^p\sum_{k=0}^{m-1}|B^\top G_j^k\eta|,
\qquad j\ge1.
\end{equation}
The dependence on $\Omega$ includes the lower bound $\lambda_1$ and
these finitely many singular-value bounds. This is the only step
that distinguishes small and large modes.

\emph{Step 3: Observation on measurable sets.}
Since $f^{(k)}(0)=(-1)^kB^\top G_j^k\eta$,
\eqref{eq:quantitative-kalman} and \eqref{eq:jet-observability} imply
\[
|\eta|\le C(1+\lambda_j)^p\ell^{-m}\int_0^\ell|f(t)|\,dt.
\]
Apply \eqref{eq:BA} on $I=(0,\ell)$ to obtain
\[
\int_0^\ell|f(t)|\,dt
\le C\left(\frac{\ell}{|E|}\right)^m\int_E|f(t)|\,dt.
\]
The factors $\ell^{-m}$ and $\ell^m$ cancel, yielding
\eqref{eq-short-time}. Both estimates are uniform over all modes.

\emph{Step 4: Arbitrary time intervals.}
Partition $(0,L)$ into $N=\lceil L\rceil$ intervals of equal length at most one.
One of them, $I=(a,b)$, satisfies $|F\cap I|\ge|F|/N$.
Applying \eqref{eq-short-time} to $e^{-aG_j}\xi$ and the translated set
$(F\cap I)-a$ gives
\[
|e^{-aG_j}\xi|\le C(1+\lambda_j)^p
\left(\frac{N}{|F|}\right)^m
\int_{F\cap I}|B^\top e^{-tG_j}\xi|\,dt.
\]
To propagate this estimate to time $L$, let $v(t)=e^{-tG_j}v_0$.
The coercivity assumption \eqref{ass:H1}, which also holds for $A^\top$,
gives
\[
\frac12\frac{d}{dt}|v(t)|^2
=-\langle A^\top v(t),v(t)\rangle
+\lambda_j^{-1}\langle M^\top v(t),v(t)\rangle
\le(\omega/\lambda_j-\alpha)|v(t)|^2.
\]
Gronwall's inequality implies
\begin{equation}\label{eq:mode-semigroup}
\|e^{-tG_j}\|\le e^{(\omega/\lambda_j-\alpha)t},\qquad t\ge0.
\end{equation}
In particular,
\[
|e^{-LG_j}\xi|
\le e^{(\omega/\lambda_j-\alpha)(L-a)}|e^{-aG_j}\xi|
\le e^{\omega L/\lambda_j}|e^{-aG_j}\xi|.
\]
Combining this bound with the preceding estimate, $N\le1+L$, and
$F\cap I\subset F$ proves \eqref{eq:remez_kalman}.
If $\lambda_j\alpha\ge\omega$, the first exponential factor is at most one,
so the propagation is contractive and no exponential loss is needed.
\end{proof}

We now use Proposition~\ref{thm:remez_kalman} to estimate the low-frequency part.
\begin{theorem}\label{thm-low-frequency}
Assume \eqref{ass:H1} and \eqref{ass:H2}. There exists
$C=C(A,M,B,m,\Omega)>0$ such that, for $0<S_1<S_2\le T$,
every measurable $E\subset(0,T)$ with $\delta=|E\cap(S_1,S_2)|>0$, every $\Lambda>0$, and every initial datum $z\in L^2(\Omega;\mathbb R^m)$,
\begin{multline}\label{eq-final-low}
\|P_\Lambda z(S_2)\|_{L^2(\Omega;\mathbb R^m)}
\\
\le\frac{C e^{\omega(S_2-S_1)}(1+\Lambda)^{p+1}(1+S_2-S_1)^m}{\delta^m}
\int_{S_1}^{S_2}\chi_E(t)
\|B^\top P_\Lambda z(t)\|_{L^2(\Omega;\mathbb R^r)}\,dt.
\end{multline}
\end{theorem}
\begin{proof}
Replace $E$ by $E\cap(S_1,S_2)$ and set $\Delta=S_2-S_1$.
For each $j$, put $D_j=-\lambda_jA^\top+M^\top$ and apply
\eqref{eq:remez_kalman} with
\[
\eta=e^{S_1D_j}\xi,\qquad L=\lambda_j\Delta,
\qquad F=\{\lambda_j(t-S_1):t\in E\}.
\]
Since $|F|=\lambda_j|E|$ and
$\lambda_j^{-1}+\Delta\le\max\{1,\lambda_1^{-1}\}(1+\Delta)$, we obtain
\begin{equation}\label{eq:single_mode}
|e^{S_2D_j}\xi|
\le\frac{C e^{\omega\Delta}(1+\lambda_j)^p\lambda_j(1+\Delta)^m}{|E|^m}
\int_E|B^\top e^{tD_j}\xi|\,dt.
\end{equation}
Define
\begin{equation}\label{eq-aj-bj}
a_j=e^{S_2D_j}z_j,\qquad b_j(t)=B^\top e^{tD_j}z_j.
\end{equation}
For $\lambda_j\le\Lambda$, \eqref{eq:single_mode} implies
$|a_j|\le H_\Lambda\int_E|b_j(t)|\,dt$, where
\[
H_\Lambda=\frac{C e^{\omega\Delta}(1+\Lambda)^{p+1}(1+\Delta)^m}{|E|^m}.
\]
Minkowski's integral inequality then gives
\[
\left(\sum_{\lambda_j\le\Lambda}|a_j|^2\right)^{1/2}
\le H_\Lambda\int_E
\left(\sum_{\lambda_j\le\Lambda}|b_j(t)|^2\right)^{1/2}\,dt.
\]
Finally, Parseval's identity yields
\begin{equation}\label{eq-norm2}
\|B^\top P_\Lambda z(t)\|_{L^2(\Omega;\mathbb R^r)}^2
=\sum_{\lambda_j\le\Lambda}|b_j(t)|^2,
\end{equation}
and the corresponding identity for $P_\Lambda z(S_2)$ proves \eqref{eq-final-low}.
\end{proof}

\subsection{Proof of Theorem \ref{thm-interpolation}}\label{sec3.2}

\begin{proof}
We use the low-frequency estimate \eqref{eq-final-low} in three steps.
Throughout this proof, $\|\cdot\|_q$ denotes the $L^q(\Omega)$ norm
with values in the appropriate Euclidean space. Set
\[
\delta=|E\cap(S_1,S_2)|,\qquad
\mathcal I(v)=\int_{S_1}^{S_2}\chi_E(t)
\|\chi_{\mathcal D_t}B^\top v(t)\|_1\,dt.
\]

\emph{Step 1: Localization.}
We first prove that
\begin{equation}\label{eq-step-2}
\|P_\Lambda z(S_2)\|_2\le K_\Lambda\mathcal I(P_\Lambda z),
\end{equation}
where
\begin{equation}
K_\Lambda=\frac{C_5(1+\Lambda)^{p+1}(1+S_2-S_1)^m
 e^{\omega S_2+C_5\sqrt\Lambda}}{\delta^m},
\end{equation}
Here $C_5>0$ depends only on the fixed parameters.
For the coefficients $b_j(t)$ defined in \eqref{eq-aj-bj}, the spectral
inequality \eqref{eq-assumption-Lebeau} gives, for almost every $t\in E$,
\begin{equation}\label{eq-906}
\sum_{\lambda_j\le\Lambda}|b_j(t)|^2
\le D e^{D\sqrt\Lambda}
\|\chi_{\mathcal D_t}B^\top P_\Lambda z(t)\|_1^2.
\end{equation}
Here we used $\sum_{\lambda_j\le\Lambda}b_j(t)e_j=B^\top P_\Lambda z(t)$.
Insert \eqref{eq-906} into \eqref{eq-final-low} and enlarge $C_5$ to
obtain \eqref{eq-step-2}.

\emph{Step 2: Frequency splitting.}
For $\theta\in(0,1)$, we seek an estimate of the form
\begin{equation}\label{eq-interpolation-step2}
\|z(S_2)\|_2\le\frac{C_7}{\delta^{m-1}}
\left(\varepsilon(\Lambda)^{-\gamma}\frac{\mathcal I(z)}\delta
+\varepsilon(\Lambda)\|z\|_2\right),
\end{equation}
where $\gamma=\theta/(1-\theta)$ and
\begin{equation}\label{eq-varepsilon-Lambda}
\varepsilon(\Lambda)=e^{-c(\Lambda-\Lambda_1)S_1/2},\qquad\Lambda>\Lambda_1.
\end{equation}
By \eqref{ass:H1} and the choice of $\Lambda_1$, the high-frequency part satisfies
\begin{equation}\label{eq-high-energy}
\|(I-P_\Lambda)z(t)\|_2\le C_{\rm h}e^{-c\Lambda t}\|z\|_2,
\qquad t\ge0,\quad\Lambda>\Lambda_1,
\end{equation}
with $C_{\rm h},c>0$ depending only on $A$ and $M$.
The triangle inequality, \eqref{eq-step-2}, and \eqref{eq-high-energy} yield
\begin{align*}
\|z(S_2)\|_2
&\le K_\Lambda\mathcal I(P_\Lambda z)
+C_{\rm h}e^{-c\Lambda S_2}\|z\|_2\\
&\le K_\Lambda\mathcal I(z)
+C_{\rm h}\bigl(\|B\||\Omega|^{1/2}K_\Lambda\delta e^{-c\Lambda S_1}
+e^{-c\Lambda S_2}\bigr)\|z\|_2.
\end{align*}
Absorbing the polynomial factor $(1+\Lambda)^{p+1}$ into
$Ce^{C\sqrt\Lambda}$, and using $\delta\le S_2$ and
$1+S_2-S_1\le1+S_2$, gives
\begin{equation}\label{eq-lebeau}
\|z(S_2)\|_2\le\frac{C_6e^{\omega S_2}(1+S_2)^m}{\delta^{m-1}}\left(e^{C_6\sqrt\Lambda}\frac{\mathcal I(z)}\delta
+e^{C_6\sqrt\Lambda-c\Lambda S_1}\|z\|_2\right),
\end{equation}
where $C_6\ge1$ depends only on the fixed parameters.
Young's inequality implies
\begin{align*}
C_6\sqrt\Lambda-\tfrac12c\Lambda S_1
&\le\frac{C_6^2}{c\theta S_1},\\*
C_6\sqrt\Lambda
&\le\tfrac12c\gamma\Lambda S_1+\frac{C_6^2}{c\theta S_1}.
\end{align*}
Moreover,
$e^{-c\Lambda S_1/2}\le\varepsilon(\Lambda)$ and
$e^{c\gamma\Lambda S_1/2}
=\varepsilon(\Lambda)^{-\gamma}e^{c\gamma\Lambda_1S_1/2}$.
Since $\gamma S_1\le S_2/(1-\theta)$, inserting these bounds into
\eqref{eq-lebeau} proves \eqref{eq-interpolation-step2} with
\begin{equation}\label{eq-C4-def}
C_7=C_6(1+S_2)^m
\exp\left(\omega S_2+\frac{C_6^2}{c\theta S_1}
+\frac{c\Lambda_1S_2}{2(1-\theta)}\right).
\end{equation}

\emph{Step 3: Optimization.}
Write $X=\|z(S_2)\|_2$, $Y=\mathcal I(z)/\delta$,
$Z=\|z\|_2$, and $\Pi=C_7/\delta^{m-1}$.
As $\Lambda$ ranges over $(\Lambda_1,\infty)$,
$\varepsilon(\Lambda)$ ranges over $(0,1)$, so
\eqref{eq-interpolation-step2} reads
\[
X\le\Pi(\varepsilon^{-\gamma}Y+\varepsilon Z),\qquad0<\varepsilon<1.
\]
Since $\delta\le S_2$ and $C_6\ge1$, we have $\Pi\ge e^{\omega S_2}$.
The energy estimate $X\le e^{\omega S_2}Z$ therefore extends this bound
to $\varepsilon\ge1$. For $Y,Z>0$, choose
$\varepsilon=(Y/Z)^{1-\theta}$ to get
\[
\|z(S_2)\|_2\le\frac{2C_7}{\delta^{m-1}}
\left(\frac{\mathcal I(z)}\delta\right)^{1-\theta}\|z\|_2^\theta.
\]
The same conclusion holds if $Y$ or $Z$ vanishes, by a limiting argument
or uniqueness. Finally, \eqref{eq-C4-def} gives
$2C_7\le C_8\exp\bigl(C_8[S_2/(1-\theta)+1/(\theta S_1)]\bigr)$
for a constant $C_8$ depending only on the fixed parameters.
This proves \eqref{eq-interpolation-inequality}.
\end{proof}

\section{Proof of Theorem \ref{theorem:obs}}\label{sec3}
We follow the telescoping argument of \cite{PhungWang,AEWZ,Fu},
using Theorem~\ref{thm-interpolation}. In this proof, $\|\cdot\|_q$
denotes the $L^q(\Omega)$ norm with values in the appropriate Euclidean space.

\emph{Step 1: Localization.}
By \cite[Lemma 2.3]{Fu}, there are $(t_0,x_0)\in\mathcal D$ and $R>0$ such that
\begin{equation}\label{eq-1047}
B_{4R}^{d+1}(t_0,x_0)\subset(0,T)\times\Omega,\qquad
|B_R^{d+1}(t_0,x_0)\cap\mathcal D|\ge\tfrac12|B_R^{d+1}|>0.
\end{equation}
Set $\mathcal D_R=B_R^{d+1}(t_0,x_0)\cap\mathcal D$ and define
\[
(\mathcal D_R)_t=\{x:(t,x)\in\mathcal D_R\},\qquad
E_R=\left\{t\in(0,T):|(\mathcal D_R)_t|\ge\frac{|\mathcal D_R|}{2T}\right\}.
\]
Since $B_{4R}^d(x_0)\subset\Omega$ and
$\mathcal D_R\subset (0,T)\times B_R^d(x_0)$, this is an observation pair
of Section~\ref{sec2}. Lemmas~\ref{lemma-Fu-1} and \ref{lemma-Fu-2}
give $|E_R|>0$ and \eqref{eq-assumption-Lebeau} with the fixed constant
$D=D(R,|\mathcal D_R|/T)$. We will also use
\begin{equation}\label{eq-subset}
\chi_{E_R}(t)\chi_{(\mathcal D_R)_t}(x)\le\chi_{\mathcal D_R}(t,x).
\end{equation}

\emph{Step 2: Interpolation.}
We recall the following density-point lemma from \cite{AEWZ,Lion,PhungWang}.
\begin{lemma}\label{lemma-density-point}
Let $E\subset(0,T)$ have positive measure and let $\ell$ be a density point of $E$.
For every $\mu>1$, there exists $\ell_1\in(\ell,T)$ such that
\[
\ell_{j+1}=\ell+\mu^{-j}(\ell_1-\ell),\qquad
\ell_j-\ell_{j+1}\le3|E\cap(\ell_{j+1},\ell_j)|,\qquad j\ge1.
\]
\end{lemma}
Fix $\beta>0$, choose $\mu=\sqrt{(\beta+2)/(\beta+1)}$, and apply the lemma
to a density point $\ell$ of $E_R$. Write $h=\ell_1-\ell$ and set
\[
\delta_j=|E_R\cap(\ell_{j+1},\ell_j)|,\qquad
\mathcal J_j=\int_{\ell_{j+1}}^{\ell_j}\chi_{E_R}(t)
\|\chi_{(\mathcal D_R)_t}B^\top z(t)\|_1\,dt.
\]
Restrict $E_R$ to $(\ell_{j+1},\ell_j)$ and translate time by $\ell_{j+2}$.
The resulting observation pair has the same spectral constant $D$,
independently of $j$. Apply Theorem~\ref{thm-interpolation} to
$z(\ell_{j+2}+\cdot)$ with $S_1=\ell_{j+1}-\ell_{j+2}$ and
$S_2=\ell_j-\ell_{j+2}$. Changing back to the original time variable yields
\begin{multline*}
\|z(\ell_j)\|_2\le\frac{C_8}{\delta_j^{m-1}}
\exp\left[C_8\left(\frac{\ell_j-\ell_{j+2}}{1-\theta}
+\frac1{\theta(\ell_{j+1}-\ell_{j+2})}\right)\right]\\
\times\left(\frac{\mathcal J_j}{\delta_j}\right)^{1-\theta}
\|z(\ell_{j+2})\|_2^\theta.
\end{multline*}
Take $\theta=\beta/(\beta+1)$. The identities for $\ell_j$ and the density estimate give
\[
(\ell_{j+1}-\ell_{j+2})^{-1}
=\frac{\mu^{j+1}}{h(\mu-1)},\qquad
\delta_j^{-1}\le\frac{3\mu^j}{h(\mu-1)}.
\]
Young's inequality therefore gives, after enlarging a constant
$C_9=C_9(C_8,\beta,h,\mu)>0$,
\begin{equation}\label{eq:iteration-young}
\|z(\ell_j)\|_2\le K\varepsilon^{-\beta}e^{C_9\mu^{j+2}}\mathcal J_j
+\varepsilon\|z(\ell_{j+2})\|_2,\qquad\varepsilon>0,
\end{equation}
where the constant
\[
K=\frac{C_\beta C_8^{\beta+1}e^{C_8(\beta+1)^2T}}
{[h(\mu-1)]^{m+(m-1)\beta}}
\]
is independent of $j$ and $z$. Polynomial factors in $\mu^j$ have been
absorbed into $e^{C_9\mu^{j+2}}$.

\emph{Step 3: Telescoping.}
Choose $\varepsilon_j=e^{-C_9\mu^{j+2}}$ in \eqref{eq:iteration-young}
and multiply by $e^{-C_9(\beta+2)\mu^j}$.
Since $(\beta+1)\mu^{j+2}=(\beta+2)\mu^j$, we obtain
\begin{equation*}
e^{-C_9(\beta+2)\mu^j}\|z(\ell_j)\|_2
-e^{-C_9(\beta+2)\mu^{j+2}}\|z(\ell_{j+2})\|_2
\le K\mathcal J_j.
\end{equation*}
Sum over even $j\ge2$. The remainder tends to zero, because
$z\in C([0,T];L^2)$ and the exponential weights tend to zero.
The integration intervals are disjoint, hence
\begin{align*}
e^{-C_9(\beta+2)\mu^2}\|z(\ell_2)\|_2
&\le K\sum_{k\ge1}\mathcal J_{2k}\\
&\le K\int_0^T\chi_{E_R}(t)
\|\chi_{(\mathcal D_R)_t}B^\top z(t)\|_1\,dt.
\end{align*}
Finally, the energy estimate gives
$\|z(T)\|_2\le e^{\omega(T-\ell_2)}\|z(\ell_2)\|_2$.
Combining this with \eqref{eq-subset} and $\mathcal D_R\subset\mathcal D$ yields
\begin{equation*}
\|z(T)\|_2\le
K e^{\omega T+C_9(\beta+2)\mu^2}
\int_{\mathcal D}|B^\top z(t,x)|\,dx\,dt,
\end{equation*}
which is \eqref{eq:observability}.
\qed

\section*{Acknowledgements}

X. Chen is grateful to Professor Gengsheng Wang and Professor Shanlin Huang for helpful discussions, particularly regarding the proof and construction of Proposition \ref{prop-Remez}.

\section*{Disclosure of AI assistance}

The authors used AI tools for proofreading and reformulating parts of
the manuscript, preparing the literature discussion and locating references,
and for assistance in simplifying their initial proof of
Proposition~\ref{thm:remez_kalman}. That initial proof was inspired by the
approach of Lissy and Zuazua \cite{LissyZuazua2019} and treated low and high
frequencies separately. The AI-assisted simplification led to the unified
observability argument presented here. The authors take full responsibility
for the mathematical content and the final text of the article.

\end{document}